\documentclass[reqno]{amsart}
\usepackage[T1]{fontenc}
\usepackage{amssymb, amsmath, color, textcomp, url,tikz}
\usetikzlibrary{matrix,arrows}
\usetikzlibrary{fit}
\usetikzlibrary{positioning}
\usepackage[labelformat=empty]{caption}

\usepackage{mdwlist}

\RequirePackage{ifpdf}
\ifpdf
   \usepackage[pdftex]{hyperref}
\else
   \usepackage[hypertex]{hyperref}
\fi

\usepackage{enumerate,xspace}

\theoremstyle{plain}
\newtheorem{theorem}{Theorem}[section]

\newtheorem{prop}[theorem]{Proposition}
\newtheorem{lemma}[theorem]{Lemma}

\theoremstyle{definition}
\newtheorem{remark}[theorem]{Remark}
\newtheorem{fact}[theorem]{Fact}
\newtheorem{definition}[theorem]{Definition}
\newtheorem{example}[theorem]{Example}

\newcommand{\nc}{\newcommand}

\nc{\Z}{\mathbb{Z}}
\nc{\Q}{\mathbb{Q}}
\nc{\N}{\mathbb{N}}
\nc{\C}{\mathbb{C}}

\nc{\M}{\mathcal{M}}
\nc\LL{\mathcal L}

\nc{\dcl}{\operatorname{dcl}}
\nc{\dclq}{\operatorname{acl^\text{eq}}}

\nc{\acl}{\operatorname{acl}}
\nc{\aclq}{\operatorname{acl^\text{eq}}}
\nc\Tq{T^\text{eq}} 
\nc{\bdd}{\operatorname{bdd}}

\nc{\tr}{\operatorname{tr.deg}}
\nc{\ldim}{\operatorname{ldim}}
\nc\inv{ ^{-1}}

\nc{\tp}{\operatorname{tp}}
\nc\cb{\operatorname{Cb}}
\nc\U{\operatorname{U}}
\nc{\cf}{\text{cf. }}
\nc{\eg}{\text{e.g. }}

\def\Ind#1#2{#1\setbox0=\hbox{$#1x$}\kern\wd0\hbox to
  0pt{\hss$#1\mid$\hss} \lower.9\ht0\hbox to
  0pt{\hss$#1\smile$\hss}\kern\wd0}
\def\Notind#1#2{#1\setbox0=\hbox{$#1x$}\kern\wd0\hbox to
  0pt{\mathchardef\nn="0236\hss$#1\nn$\kern1.4\wd0\hss}\hbox to
  0pt{\hss$#1\mid$\hss}\lower.9\ht0 \hbox to
  0pt{\hss$#1\smile$\hss}\kern\wd0}
\def\ind{\mathop{\mathpalette\Ind{}}}
\def\nind{\mathop{\mathpalette\Notind{}}}

\def\indip{\mathop{\ \ \hbox to 0pt{\hss$\mid^{\hbox to
0pt{$\scriptstyle P$\hss}}$\hss}
\lower4pt\hbox to 0pt{\hss$\smile$\hss}\ \ }}
\def\nindip{\mathop{\ \ \hbox to 0pt{\hss$\!\not{\mid}^{\hbox to
0pt{$\scriptstyle\, P$\hss}}$\hss}
\lower4pt\hbox to 0pt{\hss$\smile$\hss}\ \ }}

\begin{document}

\title{Ample Pairs}
\date{\today}

\author{Enrique Casanovas, Amador Martin-Pizarro and Daniel Palacin}
\address{Departament de Matem\`atiques i Inform\`atica;
  Universitat de Barcelona; Gran Via 585; E-08007 Barcelona; Spain}
\address{Abteilung f\"ur Mathematische Logik; Mathematisches Institut;
  Albert-Ludwig-Universit\"at Freiburg; Eckerstra\ss e 1; D-79104
  Freiburg; Germany}
\address{Einstein Institute of Mathematics; Hebrew University of
  Jerusalem; 9190401 Jerusalem; Israel}

\email{e.casanovas@ub.edu}
\email{pizarro@math.uni-freiburg.de} 
\email{daniel.palacin@mail.huji-ac.il}
\thanks{Research partially supported by the program MTM2014-59178-P}  

\begin{abstract} 
  We show that the ample degree of a stable theory with trivial
  forking is preserved when we consider the corresponding theory of
  belles paires, if it exists. This result also applies to the theory
  of $H$-structures of a trivial theory of rank $1$.
\end{abstract}

\maketitle


\section*{Introduction}
The dichotomy principle, formulated by Zilber and at the base of many
key applications of Geometric Model Theory to Diophantine Geometry,
establishes a division line on the geometry of the minimal sets in a
given theory: either the lattice of algebraically closed sets (in
$\Tq$) is modular or an algebraically closed field can be
interpreted. The dichotomy principle does not hold for strongly
minimal sets, as shown by Hrushovski, who developed a general method
\cite{eH92, eH93} to produce $\omega$-stable theories with prescribed
geometries in terms of underlying dimension functions, which agree
with Morley rank on the resulting theories. Despite the exotic
behaviour of the geometry of his \emph{ab initio} example, it
satisfies a weakening of the modularity principle, which in itself
prevents an infinite field to be interpretable \cite{aP95}. Motivated
by this, Pillay \cite{aP00} and Evans \cite{dE03} introduced the ample
hierarchy of stable theories, in order to provide finer division lines
on the analysis of the geometry of strongly minimal sets.  According
to this hierarchy, motivated by the incidence relation in the euclidean
space of the flags of affine subspaces of increasing dimension, from
one point to a hyperplane, the \emph{ab initio} construction
is of low complexity, whereas algebraically closed fields or the free
non-abelian group \cite{rS15} lie at the very top.

Little is known about preservation theorems for ampleness. In recent
work, Carmona \cite{jfC15} studied the ample degree of a sufficiently
saturated model of simple theory of rank $1$, equipped with a
distinguished predicate for a dense codense independent subset. Any
two such structures are elementarily equivalent and their common
theory is an example of an $H$-structure, as introduced by Berenstein
and Vassiliev \cite{BV16}. He showed that it is preserved whenever the
degree of ampleness is at least $2$. However, an $H$-structure of a
$1$-based theory of rank $1$ need no longer be $1$-based. This marks a
major difference with respect to Poizat's \emph{belles paires} of
models of a stable theory \cite{bP83} (or more generally, lovely pairs
of a simple theory \cite{BPV03}), which remain $1$-based if the
departing theory is \cite[Proposition 7.7]{BPV03}.

In this short note, we explore such preservation results for belles
paires. Imaginaries represent the first obstacle.  For
non-$1$-ampleness (or equivalently $1$-basedness) of belles paires,
the proof in \cite{BPV03} uses a reformulation of it, weakly linear,
which does not mention imaginaries. However, we do not know of a such
a formulation of ampleness, for degree at least $2$.  Notice that the
theory of belles paires does not have geometric elimination of
imaginaries as soon as an infinite group can be defined (or
interpreted) in the departing theory \cite{PV04}. In order to
circumvent this obstacle, we will only consider pairs of a theory with
trivial forking, which prevents the existence of definable infinite
groups. We originally thought that this assumption would only play a
minor role, in order to work with real sets in the definition of
ampleness. However triviality becomes crucial in the proofs. The question
remains thus open, whether generally the theory of belles paires
preserves ampleness.

\section{Trivialities}

From now on, fix a complete theory $T$ with in a language $\LL$. To
avoid dealing with hyperimaginaries and bounded closures, we will
assume that the theory $T$ is stable, though the statements (and their
proofs) hold for $T$ simple with the appropriate modifications.  We
work inside a sufficiently saturated (and strongly homogeneous) model
of $T$, which embeds any model of $T$ as a small elementary
substructure.

We first recall the following definitions from \cite{jG91}:

\begin{definition}\label{D:trival}
  The theory $T$ is \emph{trivial} if, whenever the tuples $a$, $b$
  and $c$ are pairwise independent over a small set of parameters $D$,
  then they are $D$-independent. Likewise, a stationary type $p$ over
  $E$ is trivial, if whenever the tuples $a$, $b$ and $c$, each
  consisting of realisations of $p\,|D$, are pairwise independent over
  $D\supset E$, then they are $D$-independent.

  The theory $T$ is \emph{totally trivial} if, whenever $a\ind_D b$
  and $a\ind_D c$, then $a\ind_D b,c$. 
\end{definition}

Although the above two notions are different, they agree whenever $T$
has finite Lascar rank \cite{jG91}. Notice that our local version of
triviality strengthens the original one from \cite{jG91}. Clearly,
local triviality is preserved under nonforking extensions and
restrictions.

\noindent Let us first remark the following easy observation:

\begin{remark}\label{R:trivial_acl}
  Given a stationary trivial type $p$ over $E$ and some
  $\Tq$-algebraically closed set of parameters $D\supset E$, suppose
  that $b$ is algebraic over $D, a$, where $a$ is a tuple of
  realisations of $p\,|D$. Then $\tp(b/D)$ is also trivial.
\end{remark}

\begin{proof}
  Let $D_1\supset D$ be given and consider three pairwise
  $D_1$-independent tuples $b_1$, $b_2$ and $b_3$ of realisations of
  the non-forking extension of $\tp(b/D)$ to $D_1$.  We may assume
  that $D_1=D$ and each $b_i$ is algebraic over $D, a_i$, where $a_i$
  is a tuple of realisations of $p\,|D$.

  By succesively taking nonforking extensions, we may assume the
  following independences hold:

\[ a_1\ind_{D, b_1} b_2, b_3\quad , \quad a_2\ind_{D, b_2} a_1, b_3 \text{
    and } a_3\ind_{D, b_3} a_2, a_1. \] 

In particular, the tuples $a_1$, $a_2$ and $a_3$ are pairwise
$D$-independent, so they are $D$-independent, as a set. Thus, so are
$b_1$, $b_2$ and $b_3$. 
\end{proof}

\begin{fact}\label{F:CNF}\cite[Corollary 5.1.8]{fW00}
Consider a stationary type $\tp(a/D)$ whose  Lascar rank has
Cantor normal form:

$$\omega^{\alpha_1}\cdot n_1+\cdots+\omega^{\alpha_k}\cdot n_k,$$
with $\alpha_1>\ldots>\alpha_k$ and $n_k\neq 0$. There are (possibly
imaginary)  elements $a=a_1,\ldots, a_k$, with $a_{i+1}$ algebraic over
$D,a_i$ for $i<k$, and  

\[ \U(a/Da_i)= \sum\limits_{j<i} \omega^{\alpha_j}\cdot n_j \text{
    and } \U(a_i/D)= \sum\limits_{j=i}^k \omega^{\alpha_j}\cdot n_j\]

\noindent In particular, the element $a_k$ is algebraic over $D, a$ and has
Lascar rank $\omega^{\alpha_k}\cdot n_k$. 
\end{fact}

Recall that, if a stationary type $p$ over $D$ has Lascar rank in
Cantor normal form

$$\omega^{\alpha_1}\cdot n_1+\cdots+\omega^{\alpha_k}\cdot n_k,$$
with $\alpha_1>\ldots>\alpha_k$ and $n_k\neq 0$, then it is
\emph{non-orthogonal} to a type of Lascar rank $\omega^{\alpha_k}$: 
there is a realisation $a$ of the  non-forking extension of $p$ to
some set $C$ and a stationary type $\tp(b/C)$ of Lascar rank
$\omega^{\alpha_k}$  such that

$$ a\nind_C b.$$

\noindent In \cite[Proposition 2]{jG91}, it is shown that a
superstable theory is trivial if and only if all the regular types in
$\Tq$ are trivial. A detailed study of the proof yields an improvement
of the above result, without assuming superstability, but solely
working with a fixed trivial type of ordinal-valued Lascar rank.
However, observe that our local definition of triviality is more
restrictive than Goode's definition. We believe the following result
is probably well-known but could not find any references:

\begin{prop}\label{P:Goode_triv}
  Let $T$ be a stable (possibly non-superstable) theory and
  $p=\tp(a/D)$ be a stationary trivial type whose Lascar rank has
  Cantor normal form:
 $$\omega^{\alpha_1}\cdot n_1+\cdots+\omega^{\alpha_k}\cdot n_k.$$
 \noindent There is some realisation $a$ of the non-forking extension
 of $p$ to some set $C$ and an imaginary element $e$ algebraic over
 $C,a$ such that $\tp(e/C)$ has Lascar rank $\omega^{\alpha_k}$.

\noindent  In particular, the type $p$ is non-orthogonal to a type of rank
$\omega^{\alpha_k}$. 
\end{prop}

\begin{proof}

  Set $n=n_k$ and $\alpha=\alpha_k$, and suppose $n\geq 2$.  Remark
  \ref{R:trivial_acl} and Fact \ref{F:CNF} allow us
  to assume that $\U(p)=\omega^{\alpha}\cdot n$. By the above, there
  is a realisation $a$ of the non-forking extension of $p$ to some set
  $C$ and a stationary type $\tp(b/C)$ of Lascar rank
  $\omega^{\alpha}$ such that
$$ a\nind_C b.$$

Set $b'=\cb(a/ C,b)$, which is not not algebraic over $C$, because of
the dependence $a\nind_C b$. Notice that $a\nind_C b'$, since
$a\ind_{C,b'} b$.  As $b'$ lies in $\aclq(C, b)$, its rank $\U(b'/C)$
is bounded by $\omega^\alpha$. If $\U(b'/C)<\omega^\alpha$, then it
contradicts the Lascar inequalities:
$$
\omega^\alpha\cdot n = \U(a/C) \leq
\U(a/C,b')\oplus \U(b'/C)<\omega^\alpha\cdot n.
$$ 

\noindent We may therefore assume that $b=b'$ is algebraic over a
finite segment of a Morley sequence of $\tp(a/C,b)$, so its type
$\tp(b/C)$ is also trivial, by Remark \ref{R:trivial_acl}.

Set $e=\cb(b/C, a)$, which lies in $\aclq(C, a)\setminus \aclq(C)$,
for $ a\nind_C b$. Thus $ a\nind_C e$. As above, a straight-forward
application of the Lascar inequalities yields that
$\U(e/C)\geq \omega^\alpha$.

Let us now show that $\U(e/C)=\omega^\alpha$, which will be done in
two steps: First, we show that $\U(e/C)<\omega^\alpha \cdot 2$. Second
we will prove the actual equality $\U(e/C)=\omega^\alpha $. 

\noindent Since $\U(e/C)\geq \omega^\alpha$, write
$\U(e/C)=\omega^\alpha+\beta$, for some ordinal $\beta$.  Choose a
finite initial segment $b_1,\ldots, b_{2m}$ of a Morley sequence of
${\rm stp}(b/C,a)$ such that $e$ is algebraic over
$b_1,\ldots,b_m$. Notice that $e$ is also algebraic over
$b_{m+1},\ldots,b_{2m}$ by indiscernibility. Thus, the sequence
$b_1,\ldots, b_{2m}$ cannot be $C$-independent, since $e$ is not
algebraic over $C$. Triviality of $\tp(b/C)$ implies that
$$b_i\nind_C b_j, \text{ whenever }i<j.$$

\noindent Hence, the Lascar inequalities
yield the following:

$$
\U(e/C) \le \U(b_1,\ldots,b_m / C) \leq
\U(b_1/C)\oplus\bigoplus_{i=1}^{m-1} \U(b_{i+1}/C,b_1,\ldots,b_{i})
<\omega^\alpha\cdot 2.
$$

\noindent Thus $\beta<\omega^\alpha$. By Fact \ref{F:CNF}, there is
some element $e'$ in $\aclq(C,e)\subset \aclq(C,a)$ such that
$\U(e'/C)=\beta <\omega^\alpha$. Since $\U(a/C)=\omega^\alpha\cdot n$,
we have that $a\ind_C e'$, so $e'$ must be algebraic over $C$, that
is, the ordinal $\beta$ is $0$. We conclude that the element $e$ has
rank $\omega^\alpha$, as desired.
\end{proof}

\section{Ampleness}

As in the previous section, let $T$ denote a complete stable theory in
a language $\LL$. We first recall the definition of $1$-basedness,
CM-triviality and $n$-ampleness \cite{aP00,dE03}:

\begin{definition}\label{D:CM}
	
  The theory $T$ is \emph{$1$-based} if for every pair of
  algebraically closed subsets $A\subset B$ in $\Tq$, and every real
  tuple $c$, we have that $\cb(c/A)$ is algebraic over
  $\cb(c/B)$. Equivalently, for every $\Tq$-algebraically closed set
  $A$ and every real tuple $c$, the canonical base $\cb(c/A)$ is
  algebraic over $c$.
	
  The theory $T$ is \emph{CM-trivial} if for every pair of
  algebraically closed subsets $A\subset B$ in $\Tq$, and every real
  tuple $c$, if $\aclq(Ac) \cap B = A$, then $\cb(c/A)$ is algebraic
  over $\cb(c/B)$.
	
  The theory $T$ is called \emph{$n$-ample} if there are $n+1$ real
  tuples satisfying the following conditions (possibly working over
  parameters):
	\begin{enumerate}
		\renewcommand{\theenumi}{\alph{enumi}}		
		\item\label{D:ample_inter}
                  $\aclq(a_0,\ldots,a_i)\cap\aclq(a_0,\ldots,a_{i-1},a_{i+1})
                  =
		\aclq(a_0,\ldots,a_{i-1})$ for every $0\leq i<n$,
              \item\label{D:ample_indep}
                $a_{i+1} \ind_{a_i} a_0,\ldots, a_{i-1}$ for every
                $1\leq i<n$,
		\item\label{D:ample_forks}  $a_n \nind a_0$.
\end{enumerate}
\end{definition}

By inductively choosing models $M_i\supset a_i$ such that

$$ M_i \ind_{a_i} M_0,\ldots,M_{i-1},a_{i+1},\ldots,a_n,$$

\noindent we can replace, in the definition of $n$-ampleness, all
tuples by models. This was already remarked in \cite[Corollary
2.5]{aP95} in the case of CM-triviality. Likewise, if the theory $T$
is $1$-based, resp. CM-trivial or $n$-ample, the corresponding
conclusion holds whenever the tuples are imaginary.

Every $1$-based theory is CM-trivial. A theory is $1$-based if and
only if it is not $1$-ample; it is CM-trivial if and only if it is not
$2$-ample \cite{aP00}. Observe that $n$-ampleness implies
$(n-1)$-ampleness. Thus, ampleness establishes a strict hierarchy
(see \cite{kT14, BMPZ14, BMPZ17}) among stable theories, according to
which both (pure) algebraically closed fields \cite{aP00} and the free
non-abelian group \cite{rS15} are $n$-ample for every natural number
$n$.

We now give an alternative characterisation of ampleness, which will
be useful in the last section:

\begin{prop}\label{P:ample_oddeven}
	The theory $T$ is $n$-ample if and only if there are $n+1$ 
	tuples satisfying the following conditions (possibly working over
	parameters):
	\begin{enumerate}
        \item\label{P:ample_inter}
          $\aclq(a_i\,,\, i\leq n \text{ even})\cap\aclq(a_i\,,\,
          i\leq n \text{ odd})= \aclq(\emptyset)$,
        \item\label{P:ample_indep}
          $a_{i+1} \ind_{a_i} a_0,\ldots, a_{i-1}$ for every
          $1\leq i<n$,
		\item\label{P:ample_forks} $a_n \nind a_0$.
	\end{enumerate}
        \noindent Furthermore, we may assume that the above tuples are
        real and enumerate small models.
\end{prop}

\begin{proof}
Suppose first that the tuples $a_0,\ldots,a_n$ witness $n$-ampleness.
They clearly satisfy conditions $(\ref{P:ample_indep})$ and
$(\ref{P:ample_forks})$, so we need only prove condition
$(\ref{P:ample_inter})$. Set $X_0=\aclq(\emptyset)$ and
$$ 
X_k=\aclq(a_i:\text{$i\le k$  even})\cap\aclq(a_i:\text{$i\le
  k$ odd}), \text{ for } 1\leq k\leq n.  
 $$

\noindent It suffices to show that $X_k=X_{k-1}$, by induction on $k$. It
clearly holds for $k=1$.  Fix $k\ge 2$, which we may assume to be
even, without loss of generality. Thus

\begin{align*} 
X_k & = \aclq(a_i:\text{$i\le k$ even}) \cap \aclq (a_i:\text{$i\le k-1$ odd}) \\ 
& \subset \acl(a_i:\text{$i\le k$ even}) \cap \aclq (a_0 \ldots a_{k-1}) \\ 
& \subset \acl(a_0 \ldots a_{k-2}a_k) \cap \aclq (a_0 \ldots a_{k-1}) \\
&\subset\aclq(Aa_0\ldots a_{k-2})                                                                                                                                                  
\end{align*} 
by condition \ref{D:CM} $(\ref{D:ample_inter})$. Both transitivity and
condition  \ref{D:CM} $(\ref{D:ample_indep})$  yield  that 
$$
a_{k-1}a_k\ind_{a_{k-2}} a_0\ldots a_{k-3},
$$  
 so $$a_k\ind_{\{a_i:\text{$i\le k-1$ even}\}} a_0\ldots a_{k-3}.$$
In particular, we have that 

\begin{align*} X_k & \subset\aclq(a_i:\text{$i\le k$
                     even})\cap\aclq(a_0\ldots a_{k-2}) \\
                   & \subset \aclq(a_i:\text{$i\le k-1$ even}).
\end{align*}
Hence $X_k\subset  \aclq(a_i:\text{$i\le k-1$ odd})\cap
\aclq(a_i:\text{$i\le k-1$ even})=X_{k-1}$, as desired. 

Suppose now the tuples $a_0,\ldots, a_n$ satisfy conditions
$(\ref{P:ample_inter})$, $(\ref{P:ample_indep})$ and
$(\ref{P:ample_forks})$. Set: 

\begin{itemize} 
\item $b_n=a_n$ and $b_{n-1}=a_{n-1}$;
\item $b_i=\aclq(a_ib_{i+1})\cap\aclq(a_ib_{i+2})$ for $0\le i\le n-2$; 
\item $A_0=\aclq(b_0)\cap\aclq(b_1)$.                                                                                                                          \end{itemize}

Notice that $$
a_n\ldots a_{i+1}\ind_{a_i} a_0\ldots a_{i-1},
$$ 
for every $1\le i<n$ by transitivity and condition
$(\ref{P:ample_indep})$. Since $a_i\subset b_i \subset
\aclq(a_{i+1}\ldots a_n )$ and $b_j\subset\aclq(a_j\ldots a_{i-1}b_i)$
for $j<i$, we have that 
$$
b_{i+1}\ind_{A_0,b_i} b_0\ldots b_{i-1}. 
$$ 
Now,  
$$
A_0\subset \aclq(a_i:\text{$i\le n$ even})\cap\aclq(a_i:\text{$i\le n$
  odd})= \aclq(\emptyset),$$ 
so  $b_n\nind b_0$, as $b_n=a_n$ and $a_0\subset b_0$, by condition
\ref{D:CM} $(\ref{D:ample_forks})$. 

We need only prove \ref{D:CM} $(\ref{D:ample_inter})$ for the
$b_i$'s. Observe that 
$$\aclq(b_ib_{i+1})\cap\aclq(b_ib_{i+2})=\aclq(b_i),$$
so in particular,
$$
\aclq(b_0)=\aclq(b_0b_1)\cap\aclq(b_0b_2).$$ 
Given $1\le i<n-1$, since $b_{i+2}b_{i+1}\ind_{b_i} b_0\ldots
b_{i-1}$, we conclude that 
$$
\aclq(b_0\ldots b_i)=\aclq(b_0\ldots b_ib_{i+1})\cap\aclq(b_0\ldots b_i b_{i+2}),
$$ 
by \cite[Fact 2.4]{aP00}.

Using a similar trick as in \cite[Remarks 2.3 and 2.5]{BMPZ14}, we can
replace the  obtained $b_i$'s by real tuples enumerating small
models: consider recursively for each $i$ a model $M_i$ containing
some  representative of $b_i$ such that 
$$
 M_i\ind_{b_i} M_0\ldots M_{i-1} b_{i+1}\ldots b_n.
$$
Clearly $M_n\nind M_0$, for each $b_i$ is contained in $M_i$. A
straightforward application of transitivity yields that
$M_{i+1}\ind_{M_i} M_0\ldots M_{i-1}$ for $1\le i<n$. 

\noindent It remains hence to see that the models $M_0,\ldots,M_n$
satisfy condition$(\ref{P:ample_inter})$.  To do so, consider some
arbitrary index $i$ with $1\le i<n$ and assume, without loss of
generality, that it is odd. It is easy to see that
$$
\aclq(M_j,b_k:j\le i, k>i \text{ odd})\cap \aclq(M_j,b_k:j\le i, k>i \text{ even})
$$
is contained in 
$$
\aclq(M_j,b_k:j< i, k\ge i \text{ odd})\cap \aclq(M_j,b_k:j\le i, k>i \text{ even}),
$$
which gives that  $\aclq(M_i:i\text{ even})\cap \aclq(M_i:i\text{
  odd}) =  \aclq(\emptyset)$, as desired.

\end{proof}

\section{Theories of Pairs}

From now on, let $T$ denote a complete stable theory in a language $\LL$. We will furthermore assume, for the sake of the presentation, that $T$ has geometric elimination of imaginaries (otherwise consider $\Tq$).

We first provide a uniform approach to both belles pairs as well as
$H$-structures of rank $1$ theories, isolating their common
features. Consider the expansion $\LL_P=\LL \cup\{P\}$ of the language
$\LL$ by a unary predicate $P$, which will be interpreted by an
infinite proper subset. Work inside a sufficiently saturated (strongly
homogenous) $\LL_P$-structure $M$, which is also a model of $T$. We
will not distinguish between $P$ and its interpretation $P^M$.

\begin{definition}
A subset $A\subset M$ is \emph{special} if $A\ind_{P\cap A} P$. 
\end{definition}

In the terminology of \cite{BPV03}, special subsets correspond to
$P$-independent subsets. Furthermore, if we interpret $P$ as a dense
independent set, in the sense of $H$-structures of a stable theory of rank
$1$ \cite{BV16}, a subset is special if and only if it contains its
$H$-basis, by minimality and the fact that the elements of $H$ are
geometrically independent.

\begin{definition}\label{D:Pair}
  A complete $\LL_P$-theory $T_P$ extending $T$ is a \emph{theory of
    pairs} of $T$ if it is stable and any sufficiently saturated (and
  strongly homogeneous) model $M$ of $T_P$ satisfies the following
  conditions:

\begin{enumerate}

\item\label{D:lovely} Given a complete $n$-ary $\LL$-type $p$ over a
  small special subset $A\subset M$, there is a realisation $b$ of $p$
  with

\[ b\ind_A P.\]

\item Two special subsets $A$ and $B$ of $M$ have the same type if and
  only if they have both the same $\LL$-type and the same
  quantifier-free $\LL_P$-type, that is, there is an $\LL$-elementary
  map which maps $A$ to $B$ and $P\cap A$ to $P\cap B$. 

\item Algebraically closed subsets in $T_P$ are special. Moreover, the
  algebraic closure in $T_P$ of a special subset $A$ coincides with
  its $\LL$-algebraic closure.

\item\label{D:forking} Non-forking independence in $T_P$ for special
  subsets $A$ and $B$ over a common $\LL_P$-algebraically closed
  substructure $C$ is characterised as follows:

$$  A\ind^P_C B \Longleftrightarrow\left\{ \begin{minipage}{6mm}
$A\ind_C B $ \\ \mbox{}\quad and \\[3mm]
$A\ind\limits_{C,P} B$ \end{minipage} \right.$$

\item\label{D:triv_EI} If $T$ is trivial, then $T_P$
  has geometric elimination of imaginaries as well. 
\end{enumerate}
\end{definition}

Notice that condition $(\ref{D:forking})$, though not explicitly
stated for $H$-structures in \cite{BV16}, is a straightforward
adaptation of the proof of \cite[Proposition 7.3]{BPV03}. Condition
$(\ref{D:triv_EI})$, which holds for belles paires (see \cite{PV04}),
is trivial for $H$-structures of theories of rank $1$, for they always
eliminate imaginaries geometrically, regardless whether $T$ is trivial
or not (see \cite[Remark 5.14]{BV16}).

A result on preservation of ampleness was obtained in \cite{jfC15} for
$H$-structures of a rank $1$ theory: for $n\geq 2$, the base theory
$T$ is $n$-ample if and only if the theory of the pair is.  In
\cite[Proposition 7.7]{BPV03}, it was shown that the theory of belles
paires of a stable theory $T$ is $1$-based whenever $T$ is. Indeed, it
suffices to show that $T_P$ is \emph{weakly $1$-based}
\cite[Definition 2.3]{BV12}: that is, given a tuple $a$ over a model
$N$, there is some $a' \models \tp_P(a/N)$ such that $a'\ind^P_N a$
and $a'\ind^P_a N$. The advantage of this formulation is that no
imaginaries appear and one reduces the question to a situation in $T$,
by using the characterisation of independence in Definition
\ref{D:Pair} $(\ref{D:forking})$.  Unfortunately, we do not know of an
imaginary-free equivalent definition for higher degrees of
ampleness. As noticed in \cite{PV04}, as soon as an infinite group is
definable in $T$, then the theory of belles paires does not have
geometric elimination of imaginaries. For particular stable theories,
such as the theory of algebraically closed fields, or more generally,
almost strongly minimal theories with infinite $\acl(\emptyset)$,
there is a suitable expansion of the language $\LL_P$ by geometric
sorts in order to obtain geometric elimination of
imaginaries. However, the characterisation of independence in
Definition \ref{D:Pair} $(\ref{D:forking})$ does not hold for
imaginary subsets.

We will now provide a result on preservation of ampleness assuming
that the theory $T$ is trivial.  Recall that  $T$ is fixed complete stable trivial theory
 with geometric elimination of imaginaries, and work inside a sufficiently
saturated (strongly homogeneous) model of an associated theory of
pairs $T_P$ of $T$, which we assume exists. 

\begin{lemma}\label{L:special_unions}
  Given special subsets $A$ and $B$  with a common
  $\LL_P$-algebraically closed substructure $C$, the following
  equivalence holds:

\[ A\ind^P_C B \Longleftrightarrow A\ind_C B. \]

\noindent Furthermore, if $T$ is totally trivial, then $A\cup B$ is again special. 
\end{lemma}

\begin{proof}
  We need only prove that $A\ind_{C,P} B$, whenever $A\ind_C B$. Since
  $A$ is special and contains $C$, we have that
  $$A\ind_{C, P\cap A, P\cap B} P.$$ Similarly, the independence
  $B\ind_{C, P\cap A, P\cap B} P$ holds. Since $A\ind_C B$, we have
  that $$A\ind_{C, P\cap A, P\cap B} B.$$ Triviality of $T$ yields
  that
\[A \ind_{C, P\cap A, P\cap B} B, P, \]
which implies $A\ind_{C,P} B$, as desired.

\noindent If $T$ is totally trivial, then the independences
$A\ind_{ P\cap A, P\cap B} P$ and $B\ind_{P\cap A, P\cap B} P$ imply
that $$A\cup B\ind_{ P\cap A, P\cap B} P,$$ so $A\cup B$ is special.
\end{proof}

\begin{theorem}\label{T:main}
  Let $T_P$ be a theory of pairs of a trivial stable theory $T$. For
  any natural number $n$, the theory $T$ is $n$-ample if and only if
  $T_P$ is.
\end{theorem}

\begin{proof}

  Suppose first that the real tuples $a_0,\ldots,a_n$ witness that $T$
  is $n$-ample over some set of parameters, which we assume to be
  empty.  By condition \ref{D:Pair} $(\ref{D:lovely})$, we may assume
  that

 \[a_0,\ldots,a_n\ind P. \]

 \noindent In particular, any subcollection of $a_0,\ldots,a_n$ is
 special, so the algebraic closures in $T_P$ and $T$ coincide. By
 Lemma \ref{L:special_unions}, so does forking (in $T_P$ and $T$). In
 particular, the tuples $a_0,\ldots,a_n$ witness that $T_P$ is
 $n$-ample.

 For the converse, suppose $T$ is not $n$-ample. By Proposition
 \ref{P:ample_oddeven}, let $a_0,\ldots, a_n$ be given with:

	\begin{enumerate}
	\renewcommand{\theenumi}{\alph{enumi}}		
      \item
        $\acl_P(a_i\,,\, i\leq n \text{ even})\cap\acl_P(a_i\,,\,
        i\leq n \text{ odd})= \acl_P(\emptyset)$,
	\item $a_{i+1} \ind^P_{a_i} a_0,\ldots, a_{i-1}$ for every $1\leq i<n$.
\end{enumerate} 
\noindent We may assume that each $a_i$ is $\LL_P$-algebraically
closed, hence special.  By Lemma \ref{L:special_unions}, we need only to
show that $a_n\ind a_0$ in $T$.

The independence $a_{i+1} \ind^P_{a_i} a_0,\ldots, a_{i-1}$ implies that 

\[\acl_P(a_{i+1},a_i) \ind^P_{a_i} \acl_P(a_0,\ldots, a_i), \]

so Lemma \ref{L:special_unions} yields that 
\[\acl_P(a_{i+1},a_i) \ind_{a_i} \acl_P(a_0,\ldots, a_i), \]

and particularly 
\[a_{i+1} \ind_{\acl_P(\emptyset),a_i} a_0,\ldots, a_{i-1}.\]

Furthermore, 
\begin{multline*}
  \acl(a_i\,,\, i\leq n \text{ even})\cap\acl(a_i\,,\, i\leq n \text{
    odd})\subset \\ \acl_P(a_i\,,\, i\leq n \text{
    even})\cap\acl_P(a_i\,,\, i\leq n \text{ odd}) =
  \acl_P(\emptyset).
\end{multline*}

Working over $\acl_P(\emptyset)$, we have that $a_n \ind a_0$, since
$T$ is $n$-ample, as desired.
\end{proof}

\begin{remark}\label{R:main}
  If $T$ is totally trivial, we can conclude that $T_P$ has the same
  degree of ampleness as $T$, without using Proposition
  \ref{P:ample_oddeven}. Indeed, if the tuples $a_0,\ldots, a_n$ are
  special and satisfy

	\begin{enumerate}
	\renewcommand{\theenumi}{\alph{enumi}}		
	\item  $\acl_P(a_0,\ldots,a_i)\cap\acl_P(a_0,\ldots,a_{i-1},a_{i+1})=
	\acl_P(a_0,\ldots,a_{i-1})$ for every $0\leq i<n$,
	\item $a_{i+1} \ind^P_{a_i} a_0,\ldots, a_{i-1}$ for every $1\leq i<n$,
\end{enumerate} 

\noindent then any subtuple $a_0,\ldots,a_i$ is again special, by
Lemma \ref{L:special_unions}, so we conclude directly that
$a_n\ind a_0$, because $T$ is not $n$-ample.
\end{remark}

We will now conclude with some examples illustrating the above result. 

\begin{example}
Let $T$ be the theory of a free pseudoplane, a (bicolored) infinite branching graph with
no (non-trivial) loops. This theory is $\omega$-stable, totally trivial not $2$-ample but $1$-ample, and has weak elimination of imaginaries \cite[Proposition 2.1]{BP00}. 
The theory $T_P$ of belles paires of $T$ is axiomatised by the following elementary properties:
\begin{itemize}
\item The universe is a free pseudoplane. 
\item Every element of $P$ has infinitely many direct neighbours in $P$, and every \emph{reduced} path between elements of $P$ is contained in $P$.
\item Given an element $a$ in a finite subset   subset $A$, there is some element $b$  connected to $a$ with $b \not\in P\cup A$. 
\end{itemize}
It is very easy to see that an $\aleph_0$-saturated model of the above theory is again a belle paire. It suffices to note that every finite set is contained in a finite superset $A$ such that $A\cup P$ is \emph{nice} in the terminology of \cite{BP00}, that is, closed under reduced paths. Thus, the theory $T_P$ is stable and is a theory of pairs in the sense of Definition \ref{D:Pair}. In particular,  Theorem \ref{T:main} yields that $T_P$ is $1$-ample but not $2$-ample.

\noindent Likewise, a similar argument yields that the theory of belles paires of the free pseudospace \cite{BP00} is $2$-ample but not $3$-ample. Furthermore, if we consider the \emph{free $n$-dimensional pseudospace} \cite{kT14, BMPZ14} (or more generally, free orthogonal buildings \cite{BMPZ17}), which are  stable $n$-ample but not $(n+1)$-ample with weak elimination of imaginaries and totally trivial, a similar axiomatisation of the theory of belles paires can be obtained. It suffices to consider the bi-interpretable structure consisting of the space of flags, where there are no reduced loops.  In particular,  Theorem \ref{T:main} yields that the theory of belles paires of free $n$-dimensional pseudospaces is  $n$-ample but not $(n+1)$-ample. 
\end{example}

\begin{example}
We will now consider a trivial but not totally trivial stable theory, which first appeared in \cite{jG91}. 

\noindent Consider the theory $T$ of a two-sorted structure $(A,B)$ with a binary function $\sigma$ from $A\times B$ to $B$ such that any $\sigma(a, \cdot)$ induces a bijection on $B$. Thus, we can talk about inverses of elements of $A$ by considering the inverse of $\sigma(a,\cdot)$. However, note that $A$ is neither closed under inverses nor has a group structure. A word of elements (and their inverses) of $A$ is \emph{reduced} if it is the empty word or does not contain a subword of the form $a \cdot a\inv$. We require that the bijection of every non-trivial reduced word has no fixed points on $B$. One can regard the second sort $B$ as a binary (directed) graph with edges labelled by elements of $A$. Namely, two elements $x$ and $y$ in $B$ are connected by an edge labelled by $a$ in $A$ if and only if $\sigma(a,x)=y$. A subset $D$ of $(A,B)$ is \emph{nice} if $D\cap A$ and $D\cap B$ are both non-empty and the path of a reduced word  between two elements of $D\cap B$ is contained in $D\cap B$, and the reduced word is contained in $D\cap A$. 

The theory of the above structure is complete, since the quantifier-free type of a nice set implies its type.  It is $\omega$-stable with weak elimination of imaginaries, and forking independence can be easily described for nice sets $D_1$, $D_2$ and $D_3$ with $D_3\subset D_1\cap D_2$: 

\[ \begin{array}[t]{ccc}
D_1\ind_{D_3} D_2 & \Longleftrightarrow & \begin{cases} \bullet\,\, D_1\cap D_2 \cap A \subset D_3 \cap A  \text{, and  }\\[2mm]
\bullet\,\,	\parbox{8cm}{Every reduced path between an element of $D_1\cap B$ and an element of $D_2\cap B$ factors through $D_3\cap B$.}\end{cases}
\end{array}\]

\noindent This theory is trivial yet not totally trivial. However, it resembles the free pseudoplane, for it is $1$-ample but not $2$-ample. Since the theory is $2$-dimensional \cite{jG91}, it does not have the finite cover property \cite[Corollary 4]{eH89}. Therefore the theory of belles paires is axiomatisable and a theory of pairs in the sense of Definition \ref{D:Pair}. By Theorem \ref{T:main}, the theory of belles paires  is $1$-ample, but not $2$-ample. 
\end{example}

\end{document}